\documentclass[12pt]{article}
\usepackage{enumerate,amsmath,amssymb,bm,ascmac,amsthm,url,caption,here}
\usepackage[dvipdfmx]{graphicx,color,lscape}
\usepackage[margin=1.3in]{geometry}

\theoremstyle{definition}
\newtheorem{thm}{Theorem}[section]
\newtheorem{thm*}{Theorem}
\newtheorem{defi}[thm]{Definition}
\newtheorem{defi*}{Definition}
\newtheorem{lem}[thm]{Lemma}
\newtheorem{lem*}{Lemma}

\newtheorem{pro*}{Proposition}

\newtheorem{rem}[thm]{Remark}

\newtheorem{claim*}{Claim}

\newcommand{\MC}[1]{\mathcal{#1}}
\newcommand{\MB}[1]{\mathbb{#1}}
\newcommand{\MF}[1]{\mathfrak{#1}}

\newcommand{\G}{\Gamma}
\newcommand{\D}{\Delta}
\newcommand{\sikaku}{\mathrel{\square}}

\title{Unification of graph products and \\ compatibility with switching}
\author{Sho Kubota}
\date{}

\begin{document}
\maketitle
\begin{abstract}
We define the type of graph products,
which enable us to treat many graph products in a unified manner.
These unified graph products are shown to be compatible with Godsil--McKay switching.
Furthermore,
by this compatibility,
we show that the Doob graphs can also be obtained
from the Hamming graphs by switching.
\vspace{10pt} 
\end{abstract}

\section{Introduction}

After the twisted Grassmann graphs were introduced by Van Dam and Koolen \cite{vDK},
these graphs were studied by many researchers 
(for example \cite{BFK, BHW, FKT})
as the first family of non-vertex-transitive distance-regular graphs
with unbounded diameter.
These graphs were originally constructed by converting a part of lines of
a point-line incidence structure,
but recently,
Munemasa \cite{M} proved that 
the twisted Grassmann graphs are actually obtained 
by Godsil--McKay switching, too.

Similarly to the ordinary Grassmann graphs and the twisted Grassmann graphs,
there are many pairs of distance-regular graphs that have the same intersection array
but they are not isomorphic to each other.
Can we obtain aimed distance-regular graphs by switching
like the twisted Grassmann graphs?
Answering this question is one of the goals in this paper.
We show that the Doob graphs can be obtained
from the Hamming graphs by switching many times.
We use compatibility with switching and the Cartesian product.
Indeed,
we can find a partition for switching on the graph after taking product and
show the isomorphism between the graph taking product after switching
and the graph switched after taking product.

Actually,
this compatibility holds not only for the Cartesian product
but also many other graph products.
In Section~3,
we consider unified graph products which are written as 
the sum of tensor products of the identity matrix and
the adjacency matrices of the original graph and its complement.
These products enable us to treat many graph products in  a unified manner.
In Section~4,
we show that compatibility with switching holds on these products,
which is another main result in this paper.
Furthermore,
this compatibility suggests the possibility that
some other pair of distance-regular graphs that have the same intersection array
can be mapped to each other by switching.
If the dual polar graphs $B_d(q)$ and $C_d(q)$ can be done so,
then we simultaneously see that $D_d(q)$ and ${\rm Hem}_{d}(q)$
can also be mapped to each other by switching.

\section{Godsil--McKay switching}

Let $\G$ be a graph and
let $\pi = \{C_1, \dots, C_t \}$ be a partition of the vertex set $V(\G)$.
The {\it characteristic matrix} of $\pi$ is the $\{0,1\}$-matrix $S$
with rows indexed by $V(\G)$ and columns indexed by $\pi$,
where
\[
S_{x, C_i} =
\begin{cases} 1 \quad \text{if $x \in C_i$,} \\ 0 \quad \text{otherwise} \end{cases}.
\]
The partition $\pi$ is called an {\it equitable partition} if for all $i,j \in [t]$,
any two vertices in $C_i$ have the same number, say $r_{ij}$, of neighbors in $C_j$.
The matrix $R = (r_{ij})$ 
is called the {\it quotient matrix} of $\pi$.
As is well known,
if $\pi$ is an equitable partition and its quotient matrix is $R$,
then $A(\G)S = SR$ holds.
Here $A(\G)$ denotes the adjacency matrix of $\G$.
Conversely,
if there exists a matrix $R$ of size $t$ such that $A(\G)S = SR$,
then $\pi$ is an equitable partition and its quotient matrix is $R$.

The following tool for constructing cospectral graphs
was introduced by Godsil and McKay \cite{GM}.
Let $[t]$ denote the set $\{1,2, \dots, t\}$.

\begin{thm}\label{GMsw} {\it 
Let $\G$ be a graph and let $\pi = \{C_1, \dots, C_t, D \}$ be a partition of $V(\G)$.
Assume that $\pi$ satisfies the following two conditions: 
	\begin{enumerate}[(i)]
	\item $\{C_1, \dots, C_t \}$ is an equitable partition of $V(\G) \setminus D$.
	\item For every $x \in D$ and every $i \in [t]$,
the vertex $x$ has either $0, \frac{1}{2}|C_i|$ or $|C_i|$ neighbors in $C_i$.
	\end{enumerate}
Construct a new graph ${\rm sw}_{\pi} \G$ by interchanging adjacency and nonadjacency 
between $x \in D$ and the vertices in $C_i$ 
whenever $x$ has $\frac{1}{2}|C_i|$ neighbors in $C_i$.
Then $\G$ and ${\rm sw}_{\pi} \G$ have the same spectrum. }
\end{thm}

This operation that transforms $\G$ into ${\rm sw}_{\pi} \G$
is called {\it Godsil--McKay switching}.
We call this partition $\pi$ used here a {\it Godsil--McKay partition}.
Also,
we call the special cell $D$ the {\it Godsil--McKay cell} of $\pi$.
In order to understand this switching more deeply,
we outline the proof of Theorem~\ref{GMsw}.

Define the matrix $Q$ indexed by $V(\G)$ as follows:
\[
Q_{xy} = \begin{cases}
\frac{2}{|C_i|} - \delta_{xy} \quad &\text{ if $x,y \in C_i$,} \\
\delta_{xy} &\text{ if $x,y \in D$,} \\
0 &\text{ otherwise. }
\end{cases} \]
Then $Q^2$ is the identity matrix and
$A({\rm sw}_{\pi} \G) = QA(\G)Q$ can be checked,
so $\G$ and ${\rm sw}_{\pi} \G$ are cospectral.

We call this matrix $Q$ used here the {\it switching matrix} with respect to $\pi$.
This matrix plays an important role in this paper.

%

Actually,
the two conditions (i) and (ii) in Theorem~\ref{GMsw}
can be written as algebraic conditions.
Let $\G$ be a graph and
$\pi = \{C_1,\dots, C_t, D\}$ be a partition of $V(\G)$.
We define
	\begin{align*}
A(\G)_{\MF{C}(\pi)} &=
A(\G)|_{(C_1 \cup \cdots \cup C_t) \times (C_1 \cup \cdots \cup C_t)}, \\
A(\G)_{\MF{D}(\pi)} &= A(\G)|_{D \times (C_1 \cup \cdots \cup C_t)},
	\end{align*}
but these are sometimes written as $A(\G)_{\MF{C}}$ and $A(\G)_{\MF{D}}$
for simplicity in the case where which partition we consider is clear.
Let $S$ be the characteristic matrix of $\pi \setminus \{D\}$.
Then $\pi$ is a Godsil--McKay partition with the Godsil--McKay cell $D$
if and only if the following two conditions hold:
	\begin{enumerate}[(i)]
	\item There exists a matrix $R$ of size $t$ such that $A(\G)_{\MF{C}}S = SR$,
	\item For any $i \in [t]$,
$(A(\G)_{\MF{D}} S)_{x, C_i} \in \{0, \frac12|C_i|, |C_i| \}$.
	\end{enumerate}

\section{Unification of graph products}

Such as the Cartesian product and the Kronecker product,
a number of graph products are known.
We give examples of such products in Table~\ref{T1}.
In this table,
$A$ and $B$ denote the adjacency matrices of original graphs.
As we see in the table,
the adjacency matrix of many graph products can be written
as the sum of tensor products of the identity matrix and
the adjacency matrices of the original graph and its complement.
We shall treat such products in a unified manner.

Let $\G$ and $\D$ be graphs.
Set $A_0 = I_{|V(\G)|}$, $A_1 = A(\G)$, $A_2 = J - I - A(\G)$,
$B_0 = I_{|V(\D)|}$, $B_1 = A(\D)$ and $B_2 = J - I - A(\D)$,
where $J$ is the all-one matrix and $I$ is the identity matrix.
And let $s_{ij} \in \{0,1\}$ for $i,j \in \{0,1,2\}$.
We consider the graph $\G \star \D$ defined by
\[ A(\G \star \D) = \sum_{i,j \in \{0,1,2\}} s_{ij} (A_i \otimes B_j). \]
Clearly,
this graph $\G \star \D$ is a simple graph
if and only if $s_{00} = 0$ holds.
We call the sequence $[0s_{01}s_{02}; s_{10 }s_{11} s_{12}; s_{20}s_{21}s_{22}]$
the {\it type} of the product.

\begin{table}[H]
\begin{center}
{\small 
\begin{tabular}{|c|c|c|c|} \hline
Name & Notation & Adjacency matrix & Type \\ \hline \hline
Cartesian product & $\sikaku$ & $(A \otimes I) + (I \otimes B)$ & $[010;100;000]$ \\ \hline
Kronecker product & $\otimes$ (or $\times$) & $A \otimes B$ & $[000;010;000]$ \\ \hline
strong product & $\boxtimes$ & $(A \otimes I) + (I \otimes B) + (A \otimes B)$ & $[010;110;000]$ \\ \hline
lexicographic product & $\circ$ & $(A \otimes J) + (I \otimes B)$ & $[010;111;000]$ \\ \hline
modular product & $\lozenge$ & 
\parbox[c]{135pt}{$(A \otimes I) + (I \otimes B) + (A \otimes B)$ \\ $+ (J-I-A) \otimes (J-I-B)$}
& $[010;110;001]$ \\ \hline
weak modular product & $\nabla$ & $(A \otimes B) + (J-I-A) \otimes (J-I-B)$ & $[000;010;001]$ \\ \hline
OR product & $\lor$ & $(A \otimes J) + (J \otimes B) - (A \otimes B)$ & $[010;111;010]$ \\ \hline
\end{tabular} 
}
\caption{Examples of graph products} \label{T1}
\end{center}
\end{table}

Note that our unification of product is a generalization of the concept of NEPS \cite{CDS}.
However, where NEPS only unifies about $2^3$ products,
our new concept unifies about $2^8$ products.


\begin{lem}\label{Lem1} {\it 
Let $\G$ be a graph with a partition $\pi = \{C_1, \dots, C_t\}$ of $V(\G)$ and
let $\D$ be a graph.
Set $C_i^{(w)} = C_i \times \{w\}$ for $w \in V(\D)$ and 
$\Pi = \{ C_i^{(w)} \mid i \in [t], w \in V(\D) \}$.
If $\pi$ is an equitable partition of $\G$,
then $\Pi$ is also an equitable partition of the product $\G \star \D$ of any type.}
\end{lem}

\begin{proof}
Let $S$ be the characteristic matrix with respect to $\pi$.
First, we remark that $S \otimes I_{|V(\D)|}$ is the characteristic matrix
with respect to $\Pi$.
Indeed,
	\begin{align*}
(x,w) \in C_i^{(w')} &\iff x \in C_i \text{ and } w = w' \\
&\iff S_{x, C_i} \cdot I_{ww'} = 1 \\
&\iff (S \otimes I)_{(x,w), (C_i, w')} = 1.
	\end{align*}
Since $\pi$ is an equitable partition,
there exists a matrix $R$ of size $t$ such that $A(\G)S = SR$.
Then, $A_0S = S$, $A_1S = SR$ and
$A_2 S = (J - I - A_1)S = S(M - I - R)$,
where $M$ is the matrix of size $t$ with $M_{ij} = |C_j|$,
so we have
	\begin{align*}
A(\G \star \D)(S \otimes I)
&= \sum_{i,j \in \{0,1,2\}} s_{ij} (A_i \otimes B_j) (S \otimes I) \\
&= \sum_{i,j \in \{0,1,2\}} s_{ij} (A_i S \otimes B_j) \\
&= \sum_{j=0}^2 \bigl( s_{0j}(S \otimes B_j) + s_{1j}(SR \otimes B_j)
+ s_{2j}(S(M-I-R) \otimes B_j) \bigr) \\
&= (S \otimes I) \sum_{j=0}^2 \bigl( s_{0j}(I \otimes B_j) + s_{1j}(R \otimes B_j)
+ s_{2j}((M-I-R) \otimes B_j) \bigr).
	\end{align*}
Therefore,
$\Pi$ is an equitable partition of the product $\G \star \D$.
\end{proof}

\section{Compatibility with switching}

Let $\MC{M}$ be the set of adjacency matrices of all finite simple graphs and
$\ast$ be a binary operation on $\MC{M}$
such that the size of $M \ast N$ is equal to the product
of the size of $M$ and $N$, for $M,N \in \MC{M}$.
Then,
we can consider a product graph $\G \ast \D$
to be the graph having the adjacency matrix $A(\G) \ast A(\D)$.
When $\G$ has a Godsil--McKay partition $\pi$,
$({\rm sw}_{\pi} \G) \ast \D$ can be defined.
Then, can we find some partition $\Pi$ for switching on the graph $\G \ast \D$?
If so,
is the graph $({\rm sw}_{\pi} \G) \ast \D$ taking product after switching
isomorphic to the graph ${\rm sw}_{\Pi} (\G \ast \D)$ switched after taking product?
In this section,
we answer these questions.
For the product $\star$ of any type considered in Section~3,
we show that there exists a Godsil--McKay partition $\Pi$ and
that the isomorphism 
$({\rm sw}_{\pi} \G) \star \D \simeq {\rm sw}_{\Pi} (\G \star \D)$ holds.

\begin{thm} \label{MT1}{\it
Let $\G$ be a graph with a Godsil--McKay partition $\pi = \{ C_1, \dots, C_t, D \}$.
Let $\D$ be a graph. Set $C_i^{(w)} = C_i \times \{w\}$ for $w \in V(\D)$,
$\MC{D} = D \times V(\D)$ and
$\Pi = \{ C_i^{(w)} \mid i \in [t], w \in V(\D) \} \sqcup \{\MC{D}\}$.
Then, $\Pi$ is a Godsil--McKay partition with the Godsil--McKay cell $\MC{D}$
on $\G \star \D$ for the product of any type.
Moreover,
$({\rm sw}_{\pi} \G) \star \D$ is isomorphic to ${\rm sw}_{\Pi} (\G \star \D)$.
}
\end{thm}

\begin{proof}
First,
we prove that $\Pi$ is a Godsil--McKay partition on $\G \star \D$.
Since $\pi$ is a Godsil--McKay partition with the Godsil--McKay cell $D$,
$\pi \setminus \{ D \}$ is an equitable partition of the graph $\G \setminus D$.
By Lemma~\ref{Lem1},
$\Pi \setminus \{ \MC{D} \}$ is an equitable partition
of the graph $(\G \setminus D) \star \D$.
Also, $(\G \setminus D) \star \D$ is nothing but $(\G \star \D) \setminus \MC{D}$,
so $\Pi \setminus \{ \MC{D} \}$ is an equitable partition of
$(\G \star \D) \setminus \MC{D}$.
Let $S$ denote the characteristic matrix with respect to $\pi \setminus \{D\}$.
Then $S \otimes I_{|V(\D)|}$ is the characteristic matrix
with respect to $\Pi \setminus \{ \MC{D} \}$.
Since $(A_0)_{\MF{D}(\pi)}$ is the zero matrix,
we have
\[ A(\G \star \D)_{\MF{D}(\Pi)}
= \sum_{i = 1}^{2} \sum_{j = 0}^{2} s_{ij} ((A_i)_{\MF{D}(\pi)} \otimes B_j) \]
and
	\begin{align}
(A(\G \star \D)_{\MF{D}(\Pi)} (S \otimes I))_{(x,w), C_l^{(w')}}
&= \sum_{i = 1}^{2} \sum_{j = 0}^{2} s_{ij} ((A_i)_{\MF{D}(\pi)} S \otimes B_j)_{(x,w), C_l^{(w')}} \notag \\
&= \sum_{i = 1}^{2} \sum_{j = 0}^{2} s_{ij} ((A_i)_{\MF{D}(\pi)} S)_{xC_l}(B_j)_{ww'}. \label{Exp2}
	\end{align}
There exists unique $j \in \{0,1,2\}$ such that
$(B_j)_{ww'} = 1$ and $(B_k)_{ww'} = 0$ for $k \in \{0,1,2\} \setminus \{j\}$, so
Expression~(\ref{Exp2}) is equal to 
$s_{1j}((A_1)_{\MF{D}(\pi)} S)_{x, C_l} + s_{2j}((A_2)_{\MF{D}(\pi)} S)_{x, C_l}$.
And $((A_2)_{\MF{D}(\pi)} S)_{x, C_l} = ((J - (A_1)_{\MF{D}(\pi)})S)_{x, C_l}
= |C_l| - ((A_1)_{\MF{D}(\pi)} S)_{x, C_l}$.
Since $D$ is a Godsil--McKay cell,
$((A_1)_{\MF{D}(\pi)} S)_{x, C_l} \in \{0 ,\frac12|C_l|, |C_l| \}$.
Thus, we have
	\begin{equation} \label{Eq1}
(A(\G \star \D)_{\MF{D}(\Pi)} (S \otimes I))_{(x,w), C_l^{(w')}} =
\begin{cases}
s_{2j} |C_l| \quad &\text{if $((A_1)_{\MF{D}(\pi)} S)_{x, C_l} = 0$,} \\
\frac{s_{1j} + s_{2j} }{2}|C_l| \quad &\text{if $((A_1)_{\MF{D}(\pi)} S)_{x, C_l} = \frac12|C_l|$,} \\
s_{1j} |C_l| \quad &\text{if $((A_1)_{\MF{D}(\pi)} S)_{x, C_l} = |C_l|$.}
\end{cases}
	\end{equation}
And $s_{ij} \in \{0,1\}$ guarantees that this value is
$0$, $\frac12|C_l^{(w')}|$ or $|C_l^{(w')}|$ for any case,
which is nothing but the condition to show,
that is, $\MC{D}$ is a Godsil--McKay cell.

Next, we show the isomorphism
$({\rm sw}_{\pi} \G) \star \D \simeq {\rm sw}_{\Pi} (\G \star \D)$.
Let $Q$ and $\tilde{Q}$ be the switching matrices
with respect to the Godsil--McKay partitions $\pi$ and $\Pi$, respectively.
Clearly,
$\tilde{Q} = Q \otimes I$ holds.
Set $A'_0 = I$, $A'_1 = A({\rm sw}_{\pi}\G)$ and $A'_2 = J - I - A({\rm sw}_{\pi}\G)$.
Since $QI_{|V(\G)|}Q = I_{|V(\G)|}$ and $QJ_{|V(\G)|}Q = J_{|V(\G)|}$,
we have $A'_i = QA_iQ$ for $i \in \{0,1,2\}$.
These imply
\begin{align*}
A({\rm sw}_{\Pi}(\G \star \D))
&= \tilde{Q} A(\G \star \D) \tilde{Q} \\
&= (Q \otimes I) A(\G \star \D) (Q \otimes I) \\
&= A(({\rm sw}_{\pi} \G) \star \D),
\end{align*}
so the isomorphism holds.
\end{proof}



By the above theorem,
we see that any graph product in Table~\ref{T1} satisfies compatibility
with Godsil--McKay switching.
Moreover,
notice that the bipartite double and the extended bipartite double
can be described as the Kronecker product
and the product of type $[000; 110; 000]$ of $\G$ and $K_2$,
respectively.
Thus,
the bipartite double of switched $\G$ is isomorphic to
the switched bipartite double of $\G$.
The same is true of the extended bipartite double.

There is one more remark.
Switching results in isomorphic graphs for products of some types,
that is, $\G \star \D = {\rm sw}_{\Pi}(\G \star \D)$ could hold.
Indeed,
considering the clique extension and the coclique extension as examples,
which are nothing but the products of type $[010; 110; 110]$ and $[010; 010; 010]$,
respectively,
\[ (A(\G \star \D)_{\MF{D}} (S \otimes I))_{(x,w), C_j^{(w')}} \in \{ 0, |C_j^{(w')}| \} \]
holds for any $(x,w) \in \MC{D}$ and for any $C_j^{(w')}$ by (\ref{Eq1}).
This means that switching does not produce a non-isomorphic graph.
Of cause,
if we take a different switching partition,
switching could produce a non-isomorphic graph.
Indeed, Abiad--Brouwer--Haemers \cite{ABH} give a switching partition
different form $\Pi$ to produce a non-isomorphic graph
as for the coclique extension.
On the other hand,
some graph products satisfy cancellation, that is,
$\G \ast \D \simeq \G' \ast \D$ implies $\G \simeq \G'$
except $\D$ is the empty graph.
(See Section~6 and 9 in \cite{HIK} for example.
The Cartesian product and the strong product satisfy cancellation in general
and the Kronecker product also satisfies in many cases.)
For these products,
$\G \star \D \not\simeq {\rm sw}_{\Pi}(\G \star \D)$ holds
if $\G \not\simeq {\rm sw}_{\pi}\G$.

\section{Distance-regular graphs obtained by Godsil--McKay switching}

As an application of Theorem~\ref{MT1},
we prove that the Doob graphs can be obtained from the Hamming graphs
by switching many times.
Not only so,
compatibility suggests the possibility that
some other pair of distance-regular graphs that have the same intersection array
can be mapped to each other by switching.

\subsection{The Hamming graphs and the Doob graphs}

The {\it Hamming graph},
denoted by $H(d,q)$,
is the Cartesian product of $d$ cliques of size $q$:
\[ H(d,q) = \underbrace{K_q \sikaku K_q \sikaku \cdots \sikaku K_q}_{d}, \]
which is known as one of examples of distance-regular graphs.
In the case $(d,q) = (2,4)$,
this graph behaves interestingly.
Writing the vertex set as $V(H(2,4)) = [4] \times [4]$ and
setting $C = \{(x,x) \mid x \in [4]\}$,
the partition $\pi = \{C, V(H(2,4)) \setminus C \}$ of $V(H(2,4))$
is a Godsil--McKay partition with the Godsil--McKay cell $V(H(2,4)) \setminus C$.
Moreover,
the switched graph ${\rm sw}_{\pi} H(2,4)$ is not isomorphic to the original graph.
This graph is known as the {\it Shrikhande graph},
denoted by $Sh$.
The {\it Doob graph}, denoted by $D(m,n)$,
is the Cartesian product of $m$ Shrikhande graphs and $n$ cliques of size $4$:
\[ D(m,n) = \underbrace{Sh \sikaku \cdots \sikaku Sh}_{m}
\sikaku \underbrace{K_4 \sikaku \cdots \sikaku K_4}_{n} \]
This graph has the same intersection array (so the same spectrum)
as the Hamming graph $H(2m+n,4)$,
but they are not isomorphic to each other
for $m \geq 1$.

Before the next theorem,
we remark that the Cartesian product satisfies commutativity and associativity
(see Section~4.2 in \cite{HIK} or we can check directly).

\begin{thm}{\it
Let $m \geq 1$ and $n \geq 0$.
The Doob graph $D(m,n)$ can be obtained from the Hamming graph $H(2m+n,4)$
by switching $m$ times.
}
\end{thm}

\begin{proof}
We prove by induction on $m$.
When $m = 1$,
compatibility implies
$D(1,n) = ({\rm sw} H(2,4)) \sikaku H(n,4) \simeq {\rm sw} (H(2,4) \sikaku H(n,4))
= {\rm sw} H(n+2,4)$.
Next,
we suppose $m > 1$.
Then we have
	\begin{align*}
D(m,n) &= \underbrace{Sh \sikaku \cdots \sikaku Sh}_{m} \sikaku H(n,4) \\
&= {\rm sw} H(2,4) \sikaku (\underbrace{Sh \sikaku \cdots \sikaku Sh}_{m-1} \sikaku H(n,4)) \tag{by associativity} \\
&\simeq {\rm sw} \left( H(2,4) \sikaku ( \underbrace{Sh \sikaku \cdots \sikaku Sh}_{m-1} \sikaku H(n,4) ) \right) \tag{by Theorem~\ref{MT1}} \\
&\simeq {\rm sw} D(m-1, n+2) \tag{by commutativity} \\
&\simeq {\rm sw} \left( \underbrace{{\rm sw} \cdots {\rm sw}}_{m-1} H(2(m-1)+(n+2),4)  \right) \tag{by induction} \\
&= \underbrace{{\rm sw} \cdots {\rm sw}}_{m} H(2m+n,4),
	\end{align*}
which is what we want to show.
\end{proof}

\subsection{Dual polar graphs and their extended bipartite double}

For the materials in this subsection,
we refer the reader to \cite{BCN} for details.
Let $V$ be one of the following spaces equipped with a specified form $f$:
\[
\begin{array}{|c|c|c|} \hline
\text{Name} &V & f \\ \hline \hline
[B_d(q)] & \MB{F}_q^{2d+1} & \text{a nondegenerate orthogonal form} \\ \hline
[C_d(q)] & \MB{F}_q^{2d} & \text{a nondegenerate symplectic form} \\ \hline
[D_d(q)] & \MB{F}_q^{2d} & \text{a nondegenerate orthogonal form of Witt index $d$} \\ \hline
\end{array}
\]
A subspace $U$ of $V$ is called {\it totally isotropic}
if $U \subset U^{\perp}$,
where $U^{\perp} = \{ v \in V \mid f(v,u) = 0 \, (\forall u \in U) \}$.
Note that maximal totally isotropic subspaces have dimension $d$.
The {\it dual polar graph} (on $V$)
has the maximal totally isotropic subspaces as vertices and 
two subspaces $W_1$, $W_2$ are adjacent if and only if $\dim W_1 \cap W_2 = d-1$.
We denote the graphs defined by $[B_d(q)]$, $[C_d(q)]$, $[D_d(q)]$ as
$B_d(q)$, $C_d(q)$, $D_d(q)$, respectively.
These graphs are also known as examples of distance-regular graphs and
$B_d(q)$ and $C_d(q)$ have the same intersection array.
Moreover,
they are isomorphic to each other only if $q$ is even.
Furthermore,
the extended bipartite double of $B_d(q)$ and $C_d(q)$ are again
distance-regular graphs \cite{BHe}.
That of $B_d(q)$ is $D_{d+1}(q)$ and
that of $C_d(q)$ is the {\it Hemmeter graph} with diameter $d+1$,
denoted by ${\rm Hem}_{d+1}(q)$.
And $D_{d+1}(q)$ and ${\rm Hem}_{d+1}(q)$
are not isomorphic to each other if $q$ is odd.
Here recall that the extended bipartite double of a graph $\G$ can be regard as
the product of type $[000; 110; 000]$ of $\G$ and $K_2$.
Thus,
if $B_d(q)$ and $C_d(q)$ can be mapped to each other
by switching (maybe twice or several times?),
then we simultaneously see that $D_{d+1}(q)$ and ${\rm Hem}_{d+1}(q)$
can be mapped to each other by switching.
The following diagram describes our argument here.

\renewcommand{\arraystretch}{1.5}
\[
	\begin{array}{ccc}
B_d(q) & \xrightarrow{\text{extended bipartite double}} & D_{d+1}(q) \\
\text{switching?} \downarrow \hspace{40pt}& & \hspace{40pt} \downarrow \text{switching!} \\
C_d(q) & \xrightarrow{\text{extended bipartite double}} & {\rm Hem}_{d+1}(q)
	\end{array}
\]

\section*{Acknowledgements}

The author would like to thank Akihiro Munemasa for helpful advice
and Tomonori Hashikawa for valuable comments on Subsection~5.2.


\begin{thebibliography}{99}

\bibitem{ABH} A. Abiad, A.E. Brouwer and W.H. Haemers,
{\it Godsil--McKay switching and isomorphism},
Electron. J. Linear Algebra 28 (2015), 4-11.

\bibitem{BCN} A.E. Brouwer, A.M. Cohen and A. Neumaier,
\emph{Distance-Regular Graphs},
Springer-Verlag, Heidelberg, 1989.

\bibitem{BFK} S. Bang, T. Fujisaki and J.H. Koolen,
\emph{The spectra of the local graphs of the twisted Grassmann graphs},
European J. Combin. 30 (2009), 638--654.

\bibitem{BHW}
Benjian Lv, Li-Ping Huang and Kaishun Wang,
\emph{Endomorphisms of Twisted Grassmann Graphs},
Graphs Combin. 33 (2017), 157--169.

\bibitem{BHa} A.E. Brouwer, W.H. Haemers, 
\emph{Spectra of Graphs}, Springer, New York, 2012.

\bibitem{BHe} A. Brouwer, J. Hemmeter,
\emph{A new family of distance-regular graphs and
the $\{0, 1, 2\}$-cliques in dual polar graphs},
European J. Combin. 13 (1992), 71--79.

\bibitem{CDS} Drago\v{s} M. Cvetkovi\'{c}, M. Doob, S. Horst,
\emph{Spectra of graphs},
VEB Deutscher Verlag der Wissenschaften, Berlin, 1982.

\bibitem{vDK} E.R. van Dam and J.H. Koolen,
\emph{A new family of distance-regular graphs with unbounded diameter},
Invent. Math. 162 (2005), 189--193.

\bibitem{FKT} T. Fujisaki, J.H. Koolen and M. Tagami,
\emph{Some properties of the twisted Grassmann graphs},
Innov. Incidence Geom. 3 (2006), 81--87.

\bibitem{GM} C.D. Godsil, B.D. McKay, 
\emph{Constructing cospectral graphs}, 
Aequationes Math. 25 (1982), 257--268.

\bibitem{GR} C.D. Godsil and G. Royle, \emph{Algebraic Graph Theory}, 
Graduate Texts in Mathematics, vol. 207, Springer, New York, 2001.

\bibitem{HIK} R. Hammack, W. Imrich and S. Klav\v{z}ar,
\emph{Handbook of Product Graphs},
2nd Ed. (CRC Press, Boca Raton, 2011).

\bibitem{M} A. Munemasa,
\emph{Godsil--McKay switching and twisted Grassmann graphs},
Des. Codes Cryptogr, to appear.

\bibitem{MT} A. Munemasa and V.D. Tonchev,
\emph{The twisted Grassmann graph is the block graph of a design},
Innov. Incidence Geom. 12 (2011), 1--6.

\end{thebibliography}
\end{document}